\documentclass[12pt]{article}

\usepackage{fullpage}
\usepackage{amssymb}
\usepackage{amsmath}
\usepackage{amsthm}
\usepackage{algorithmic}

\usepackage[english]{babel}
\theoremstyle{plain}
\newtheorem{theorem}{Theorem}
\newtheorem{corollary}[theorem]{Corollary}

\begin{document}

\title{The number of multinomial coefficients based on a set of partitions of $n$ into $k$ parts and divided by $k$ evenly(A200144)}
		 
\date{}
\author{Dmitry Kruchinin\\
Tomsk State University of Control Systems and Radioelectronics,\\
 Russian Federation\\
kdv@keva.tusur.ru}
\maketitle

\begin{abstract}
In this paper we obtained an original integer sequence based on the properties of the multinomial coefficient.
We investigated a property of the sequence that  shows connection with a primality testing.
For any prime $n$ the $n$-th term  in the sequence is less by 1 than the number of partitions of $n$.
We hypothesize the existence of an asymptotic algorithm of primality testing.
\end{abstract}

\section{Introduction}
In this paper we study the properties of the integer sequence 
The sequence A200144\cite{oeis} is given by the number of multinomial coefficients based on a set of partitions of $n$ into $k$ parts such that the multinomial coefficient is divided by $k$ evenly.

Now we shall give the following notation.
Let $c_k=\{\lambda_1,\lambda_2,\ldots,\lambda_k\}$ be  a composition of a natural number $n$ into $k$ parts, where $\lambda_i>0, i=1..k$ $\{\lambda_1+\lambda_2+\cdots+\lambda_k=n\}$. Then by $C_n$ denote a set of all compositions of $n$ into $k$ parts.
Let $\pi_k=\{\lambda_1,\lambda_2,\ldots,\lambda_k\}$ be  a partition of a natural number $n$ into $k$ parts, where $\lambda_i>0, i=1..k$ $\{\lambda_1+\lambda_2+\cdots+\lambda_k=n\}$. Then by $P_n$ denote a set of all partitions of $n$ into $k$ parts.

The multinomial coefficient based on the partition $\pi_k$ is 
\begin{equation}
\label{multi}
b(\pi_k)=\binom{k}{j_1,j_2,\ldots ,j_m}=\frac{k!}{j_1!j_2!\cdots j_m!},
\end{equation} 
where $j_i$ is a number of equal $\lambda_i$ in the partition $\pi_k$. 

Let us consider   the method of obtaining the sequence A200144\cite{oeis} and its properties.
\section{Main Theorems}
\begin{theorem}
The sum
\begin{equation}
		\label{sum1}
			\sum_{k=1}^n\frac{n\cdot p_k}{k}\sum_{c_k\in C_n} a_{\lambda_1}a_{\lambda_2}\cdots a_{\lambda_k} 
		\end{equation} 
	is an integer for any integer sequences $a_1, a_2, \ldots, a_n$  and $p_1, p_2, \ldots, p_n$.
\end{theorem}

\begin{proof}
Let us construct a generating function $F(x)=a_1x+a_2x^2+\cdots+a_nx^n$. Then 
$$
F^{\Delta}(n,k)=\sum_{c_k\in C_n} a_{\lambda_1}a_{\lambda_2}\cdots a_{\lambda_k}
$$
is the coefficient of the powers of the generating function $F(x)$\cite{KruchininVV_2010}.

Let us construct a generating function $P(x)=p_1+p_2x+\cdots+p_nx^{n-1}$ with integer coefficients. 
Integrating $P(x)$ in $x$, we obtain the generating function 
$$R(x)=\frac{p_1}{1}x+\frac{p_2}{2}x^2+\cdots+\frac{p_n}{n}x^n,$$
where the constant term is zero.

For the superposition of generating functions  $G(x)=R\left(F(x)\right)$  the function of coefficients has the form
$$
g(n)=\sum_{k=1}^n \frac{p_k\cdot F^{\Delta}(n,k)}{k},
$$
$$
G(x)=\sum_{n>0} g(n)x^n.
$$

If we consider derivative $G'(x)$
we obtain the following expression
\begin{equation}
\label{derivative}
G'(x)=\left( R\left(F(x)\right)\right)'=F'(x)P\left(F(x)\right)
\end{equation}
Let us consider the expression (\ref{derivative}) as a product of the generating functions $F'(x)$ and the superposition $P\left(F(x)\right)$. The coefficients of $F'(x)$ are integers. The coefficients of the superposition of generating functions $P\left(F(x)\right)$ are also integers by virtue of the fact that
$F(x),P(x)$ are generating functions with integer coefficients.
The product of functions with integer coefficients  also has integer coefficients.

Therefore, the coefficients of the generating function
$$
G'(x)=g_1+2g_2x^1+\cdots+ng_nx^{n-1}
$$
are integers.

Hence the value of the expression 
$$
ng(n)=\sum_{k=1}^n \frac{n\cdot p_k\cdot F^{\Delta}(n,k)}{k}=\sum_{k=1}^n\frac{n\cdot p_k}{k}\sum_{c_k\in C_n} a_{\lambda_1}a_{\lambda_2}\cdots a_{\lambda_k} 
$$
is an integer for any integer sequences $a_1, a_2, \ldots, a_n$  and $p_1, p_2, \ldots, p_n$.
\end{proof}

The basis of the formula (\ref{sum1}) is the compositions of a natural number $n$ into $k$ parts.

Using a connection between compositions and partitions, consider the integer expression (\ref{sum1}) by multinomial coefficient:
\begin{equation}
\label{sum2}
\sum_{k=1}^n\frac{n\cdot p_k}{k}\sum_{c_k\in C_n} a_{\lambda_1}a_{\lambda_2}\cdots a_{\lambda_k} =
\sum_{k=1}^n\frac{n\cdot p_k}{k}\sum_{\pi_k \in P_n} b(\pi_k)\cdot a_{\lambda_1}a_{\lambda_2}\cdots a_{\lambda_k},	
\end{equation}	
where $p_k, a_{\lambda}$ are integers.

\begin{theorem}
The expression $\frac{n\cdot b(\pi_k)}{k}$ is an integer for all partitions $\pi_k\in P_n$.
\end{theorem}

\begin{proof}
Suppose that there exists the non-integer term $s_{k_0}$ in the expression (\ref{sum2}), where 

$$s_k=\frac{n\cdot p_k}{k}\sum_{\pi_k \in P_n} b(\pi_k)\cdot a_{\lambda_1}a_{\lambda_2}\cdots a_{\lambda_k}.$$

Since $p_k$ are any integers, we equate all the $p_k$ to zero, except $p_{k_0}$.
Then there exists a unique term  $s_{k}$ in the expression (\ref{sum2}) such that $s_k=s_{k_0}$ is  a non-integer term. It follows that the expression (\ref{sum2}) is not an integer. That is false.

Therefore, the value of expression
\begin{equation}
\label{sum3}
\frac{n}{k}\sum_{\pi_k \in P_n} b(\pi_k)\cdot a_{\lambda_1}a_{\lambda_2}\cdots a_{\lambda_k},	
\end{equation}
is an integer for any $k$.

Now we choose an arbitrary partition $\pi_k$. Fix an alphabet of values $\lambda_i$ such that $\lambda_i \in  \pi_k$.
Since $a_i$ are any integers, we equate all the $a_{\lambda_i}$ to zero such that $\lambda_i$ do not belong  to the alphabet.

So we get the integer expression:
\begin{equation}
\label{sum4}
\frac{n}{k}\sum_{\pi_k \in P_n} b(\pi_k)\cdot a_{\lambda_1}^{j_{\lambda_1}}\cdots {a_{\lambda_m}}^{j_{\lambda_m}}, 
\end{equation}
where $\pi_k=\{\lambda_1,\lambda_2,\ldots,\lambda_k\}$ consists of the fixed values of $\lambda_i$; 
$j_{\lambda_1}+\cdots+j_{\lambda_m}=k$; $j_{\lambda_i}>1$, because if there exists at least one $j_{\lambda_i}=1$, then the expression $\frac{b(\pi_k)}{k}$ is always integer.

Now we consider the following equation:
\begin{equation}
\label{Thue}
\alpha_0x^n+\alpha _1x^{n-1}y+\alpha _2x^{n-2}y^2+\cdots +\alpha_ny^n=k\cdot c
\end{equation}
where $n>2$ is integer and $\alpha_0,\alpha_1,\alpha_2,\ldots ,\alpha_n,c$ are also integer.

According the Thue--Siegel--Roth theorem\cite{Davenport}, the equation (\ref{Thue}) has only a finite number of solutions in integers $x$ and $y$.

Let us get the expression (\ref{sum4}) in the form (\ref{Thue}).
We choose one parameter  $a_{\lambda_1}$ such that the degree of $a_{\lambda_1}$ is different in the different term of the expression (\ref{sum4}) and equate the other $a_{\lambda_i}$ to $1$.

So we get 
\begin{equation}
\label{suma1}
\sum_{\pi_k \in P_n}n\cdot b(\pi_k)\cdot a_{\lambda_1}^{j_{\lambda_1}}=k\cdot c.
\end{equation}
where $x=1$, $y=a_{\lambda_1}$, $\alpha_i=n\cdot b(\pi_k)$, $c$ is an integer. 

Therefore, according \cite{Davenport}, in the general case, the divisibility is due to the coefficients $\alpha_i$.
Hence, the expression $\frac{n\cdot b(\pi_k)}{k}$ is integer for all partitions of a natural number $n$ into $k$ parts.
\end{proof}

\begin{corollary}
\label{Cor}
If $n$, $k$ are relatively prime, then for all partitions $\pi_k \in P_n$ the number of parts $k$ evenly divides the multinomial coefficient $b(\pi_k)$.
\end{corollary}

If $n$, $k$ are not relatively prime and $\gcd(n,k)=k$, then
we can share the $n$ to the sum of equal $\lambda_i$. Therefore, the number of equal $\lambda_i$ in the partition $\pi_k$ is equal to $k$.
Then the multinomial coefficient is equal to one: $b(\pi_k)=\frac{k!}{k!}=1$. The divisibility is due to the multiplication by $n$.

\section{Generation of the sequence A200144\cite{oeis}}

Let us consider a construction of integer sequence such that is given by the number of multinomial coefficients based on a set of partitions of $n$ into $k$ parts, where the multinomial coefficient is divided by $k$ evenly.

For counting the number of multinomial coefficients such that the multinomial coefficient is divided by $k$ evenly, we use the following algorithm.

\begin{algorithmic}[1]
\label{algo}
\REQUIRE $n$;
\FOR {$k=1,\ldots,n$}
	\FORALL {partitions of $n$ into $k$ parts}
		\FORALL {different values of $\lambda_i$ in the partition $\pi_k$}
		\STATE count $j_i$ such that $j_i$ is the number of equal $\lambda_i$ in the partition $\pi_k$;
		\ENDFOR
		\STATE calculate the multinomial coefficient $b(\pi_k)$, according to the formula (\ref{multi});
		\STATE divide by $k$;
		\IF {result is integer} 
		\STATE increase the number of multinomial coefficients such that the multinomial coefficient is divided by $k$ evenly: $d:=d+1$; 
		\ENDIF
	\ENDFOR
\ENDFOR
\PRINT $d$.
\end{algorithmic}

For generating the set of all partitions of $n$ into a fixed number of parts we can use Algorithm H (partition into k parts) developed in C.~F.~Hindenburg's dissertation\cite{Knut4}.

\textbf{Example 1.}
Let us consider the algorithm (\ref{algo}) with a small value of $n$ for one value of $k$ that is the steps from 2 to 11.

Let $n$ be $7$.

Then the set of partitions of $n$ into $k=4$ parts is
$$[1,1,1,4];
[1,1,2,3];
[1,2,2,2].$$

Therefore, the number $j_i$ of different parts in each partition respectively is
$$[3,1];
[2,1,1];
[1,3].$$

The multinomial coefficients $b(\pi_k)$ divided by k respectively are 
$$\frac{4!}{4\cdot(3!\cdot 1!)}=1,\qquad d:=d+1;$$
$$\frac{4!}{4\cdot(2!\cdot 1!\cdot 1!)}=2,\qquad d:=d+1;$$
$$\frac{4!}{4\cdot(1!\cdot 3!)}=1,\qquad d:=d+1.$$

Using the algorithm (\ref{algo}), we obtain a original integer sequence. 
The first 20 elements of the sequence are shown below:
\begin{equation}
\label{seq}
1, 1, 2, 3, 6, 7, 14, 17, 27, 34, 55, 64, 100, 121, 167, 213, 296, 354, 489, 594,\ldots
\end{equation}
The sequence is registered in  the online encyclopedia of integer sequences with the number A200144\cite{oeis}.

Modeling the sequence (\ref{seq}), we observe a monotonic increase. Hence, let us formulate the following hypothesis.

\textbf{Hypothesis 1.}
The integer sequence A200144\cite{oeis} is  monotonically increasing.

\section{Property of the sequence A200144\cite{oeis}}
To show a property of the sequence A200144\cite{oeis}, we consider the sequence A000041\cite{oeis} that is the number of partitions of a natural number $n$.
\begin{equation}
\label{seq1}
1, 2, 3, 5, 7, 11, 15, 22, 30, 42, 56, 77, 101, 135, 176, 231, 297, 385, 490, 627,\ldots
\end{equation}

Comparing the sequence (\ref{seq}) with the sequence A000041\cite{oeis}, we see that for any prime $n$ the $n$-th term  in the sequence (\ref{seq}) is less by 1 than in the sequence (\ref{seq1}) respectively.

The reason is that for primes $n$ $\gcd(n,k)=1$, where $k<n$.
Hence, all the   multinomial coefficients based on the set of partitions of $n$ into $k$ parts are divided by $k$ evenly, except  the case $k=n$: $b(\pi_n)=1$.

An important result is that for primes $n$ the number of multinomial coefficients based on the set of partitions of $n$ into $k$ parts, divided by  $k$ evenly, is less by 1 than the number of partitions of a natural number $n$.
Thus,there is a connection between the partitions of a natural numbers and a primality testing.

\section{Conclusion}
Researches, associated with partitions, show that for a monotonically increasing sequence, associated with partitions, there exists an asymptotic method for calculating terms.
This enables to make the following hypothesis that for integers there exists an algorithm of asymptotic calculating the number of partitions such that the multinomial coefficient is divided by the number of parts in the partition evenly.
If such an expression is found, we will obtain the asymptotic algorithm of primality testing.

2000 Mathematics Subject Classification: 
Primary 11P81; Secondary 11B65, 11A41, 05A10.

Keywords: multinomial coefficient, partition, integer sequence, divisibility, primes.

\end{document}